\documentclass[12pt,reqno]{amsart}
\usepackage[a4paper,margin=2.4cm]{geometry}
\usepackage{amssymb,amsmath,amsthm,mathtools}
\usepackage{enumitem}
\usepackage{cases}

\usepackage[colorlinks=true, linkcolor=blue, urlcolor=magenta, hypertexnames=false]{hyperref}

\newcommand{\N}{{\mathbb N}}
\newcommand{\Z}{{\mathbb Z}}
\newcommand{\R}{{\mathbb R}}
\newcommand{\C}{{\mathbb C}}

\newcommand{\dd}{\mathrm d}
\newcommand{\loc}{\mathrm{loc}}

\numberwithin{equation}{section}

\theoremstyle{plain}
\newtheorem{theorem}{\bf Theorem}[section]
\newtheorem{lemma}[theorem]{\bf Lemma}
\newtheorem{proposition}[theorem]{\bf Proposition}

\theoremstyle{definition}
\newtheorem{definition}[theorem]{\bf Definition}

\theoremstyle{remark}
\newtheorem{remark}[theorem]{\bf Remark}

\newcommand{\grad}{{\nabla}}
\newcommand{\lap}{{\Delta}}

\newcommand{\abs}[1]{\left|#1\right|}

\begin{document}

\title[Optimal fractional discrete Hardy weights]{Optimal Hardy weights for the fractional discrete Dirichlet Laplacian on the half-line}

\author{Franti\v{s}ek \v{S}tampach}
\address{Department of Mathematics, Faculty of Nuclear Sciences and Physical Engineering, Czech Technical University in Prague, Trojanova 13, 12000 Prague~2, Czech Republic.}
\email{stampfra@cvut.cz}

\author{Jakub Waclawek}
\address{Department of Mathematics, Faculty of Nuclear Sciences and Physical Engineering, Czech Technical University in Prague, Trojanova 13, 12000 Prague~2, Czech Republic.}
\email{waclajak@cvut.cz}

\subjclass[2020]{Primary 26D15; Secondary 39A12, 47B39}

\keywords{discrete Hardy inequality, fractional Hardy inequality, optimality}

\date{\today}

\begin{abstract}
For positive powers of the discrete Dirichlet Laplacian $(-\lap_{\mathrm{D}})^{\sigma}$, non-trivial Hardy weights exist if and only if $0<\sigma<3/2$. For $0<\sigma\leq 1$, an explicit optimal Hardy weight for $(-\lap_{\mathrm{D}})^\sigma$ was obtained by
Das and de la Fuente-Fern{\'a}ndez. Their ground state argument does not apply when
$1<\sigma<3/2$, because the off-diagonal matrix entries of $(-\lap_{\mathrm{D}})^\sigma$ are no longer all non-positive.
Nevertheless, we show that their result is true for the entire subcriticality range $0<\sigma<3/2$. For $1<\sigma<3/2$,
our argument combines the factorisation
\[
(-\lap_{\mathrm{D}})^{\sigma}=\grad_{\mathrm{D}}^{*}(-\lap_{\mathrm{N}})^{\sigma-1}\grad_{\mathrm{D}},
\]
where $\grad_{\mathrm{D}}$ and $-\lap_{\mathrm{N}}$ are the discrete Dirichlet gradient and the discrete Neumann Laplacian, respectively, with a suitable two-term ground state representation. In addition, we use these arguments to derive another explicit family of optimal Hardy weights for any $\sigma\in(0,3/2)$.
\end{abstract}

\maketitle

\section{Introduction and main results}

\subsection{Introduction}

It was proved in~\cite{ger-kre-sta_25} that a positive power of the
discrete Dirichlet Laplacian $-\lap_{\mathrm{D}}$ on the half-line is
subcritical if and only if its exponent lies in $(0,3/2)$. More precisely,
the non-negative operator $(-\lap_{\mathrm{D}})^\sigma$, defined by the
functional calculus for any $\sigma>0$, admits a non-zero non-negative
sequence $W$ such that
\[
(-\lap_{\mathrm{D}})^\sigma\geq W
\]
in the sense of quadratic forms on $\ell^2(\N)$ if and only if
$0<\sigma<3/2$. Here $W$ is identified with the corresponding
multiplication operator and is called a \emph{Hardy weight} for
$(-\lap_{\mathrm{D}})^\sigma$.

Hardy weights for $(-\lap_{\mathrm{D}})^\sigma$ throughout
$\sigma\in(0,3/2)$ were obtained in~\cite{ger-kre-sta_25}.
More recently, Das and de la Fuente-Fern{\'a}ndez
found an explicit family $W_\sigma^{\mathrm{op}}$ and
proved its optimality for $0<\sigma\leq1$;
see~\cite{das-fer_26} and~\eqref{eq:wop_explicit} below.
Their method relies on the
criticality theory for graph Schr{\"o}dinger operators developed
in~\cite{kel-pin-pog_20}, see also~\cite{kel-pin-pog_18}, which requires the off-diagonal matrix entries of $(-\lap_{\mathrm{D}})^\sigma$ to be non-positive. This condition fails
for $\sigma>1$, restricting their approach to $\sigma\in(0,1]$.
The aim of the present article is to extend their result to
$1<\sigma<3/2$, thereby covering the full subcriticality range of $(-\lap_{\mathrm{D}})^\sigma$. In addition, we deduce
a new family of optimal Hardy weights, distinct from the preceding family for all $\sigma\in(0,3/2)\setminus\{1/2\}$.

A similar obstruction arises in~\cite{sta-wac_26} for a different
half-line discrete fractional Laplacian $(-\lap)^\sigma$, defined by compressing a
power of the full-line discrete Laplacian to $\ell^2(\N)$. The latter
Laplacian is the Laurent operator defined by
\[
(-\lap_{\Z}u)_n=-u_{n-1}+2u_n-u_{n+1},\qquad n\in\Z,
\]
on $\ell^2(\Z)$, and its compressed powers are $(-\lap)^\sigma:=P_+^*(-\lap_{\Z})^\sigma P_+$,
where $P_+:\ell^2(\N)\to\ell^2(\Z)$ is the canonical embedding
given by extension by zero. The operator $(-\lap)^\sigma$ is Toeplitz and
differs from $(-\lap_{\mathrm{D}})^\sigma$ by a Hankel operator;
see~\cite[Remark~2.3]{sta-wac_26}. Unlike
$(-\lap_{\mathrm{D}})^\sigma$, it is subcritical for every $\sigma>0$.
However, its off-diagonal matrix entries also fail to be
non-positive when $\sigma>1$, so the graph Schr{\"o}dinger
operator theory does not apply directly either. Nevertheless, the authors
of~\cite{sta-wac_26} obtained the explicit optimal Hardy weights
\begin{equation}
\label{eq:sta-wac_opt_weight}
W(n)=\frac{4^{\sigma}\Gamma^2(\sigma+1/2)}{\pi}\frac{\Gamma(2n-1)}{\Gamma(2n-1+2\sigma)}, \quad n\in\N,
\end{equation}
for all $\sigma>0$ by deriving a more elaborate ground state representation
with non-negative remainder terms; see~\cite[Theorem~3.4]{sta-wac_26}.

We apply a similar idea here, replacing the direct ground state
representation of $(-\lap_{\mathrm{D}})^\sigma$ by one suited to proving
both the fractional Hardy inequality and optimality; compare
\cite[Proposition~3.4]{das-fer_26} with
Proposition~\ref{prop:two-term-repre} below.
An additional ingredient is the factorisation
\[
(-\lap_{\mathrm{D}})^\sigma
=\grad_{\mathrm{D}}^*(-\lap_{\mathrm{N}})^{\sigma-1}\grad_{\mathrm{D}},
\]
used for $1<\sigma<3/2$; see Lemma~\ref{lem:factorisation}.
It reduces the non-local part of the argument to the Neumann
fractional Laplacian $(-\lap_{\mathrm{N}})^{\sigma-1}$, whose
off-diagonal matrix entries are strictly negative. This allows
us to obtain a two-term ground state representation with
non-negative remainders. A general non-local ground state method for fractional Hardy inequalities, which provides a methodological background for this paper, was
developed by Frank and Seiringer in~\cite{fra-sei_08}.

We briefly recall related results in the continuum, where several
distinct fractional Laplacians can be defined on the positive half-line
$\R_+$. Besides the spectral power of the Dirichlet Laplacian, one can
consider the regional fractional Laplacian, whose quadratic form
involves integration over $\R_+\times\R_+$, and the restricted
fractional Laplacian, obtained by extending functions by zero to $\R$
and using the quadratic form of the full-line fractional Laplacian.
These operators should be distinguished.

The regional and restricted realisations of the half-line fractional Laplacian were studied by Bogdan and Dyda~\cite{bog-dyd_11} for $0<\sigma<1$.
The authors established the sharp Hardy inequality for the regional fractional Laplacian with the
weight
\[
W(x)=
\frac{\Gamma^2(\sigma+1/2)-\Gamma(2\sigma)\sin(\pi\sigma)}{\pi}
\frac{1}{x^{2\sigma}},
\qquad x>0.
\]
The coefficient is positive for $\sigma\neq1/2$ and vanishes at
$\sigma=1/2$. Related fractional boundary Hardy inequalities in the
critical case $\sigma=1/2$, involving logarithmic corrections, were recently
studied in~\cite{roy-sahu_26}. For the restricted fractional Laplacian,
Bogdan and Dyda~\cite{bog-dyd_11} obtained the sharp weight
\[
W(x)=\frac{\Gamma^2(\sigma+1/2)}{\pi}\frac{1}{x^{2\sigma}},
\qquad x>0,
\]
whose coefficient is positive throughout $\sigma\in(0,1)$.

On the full line, both the discrete fractional Laplacian
$(-\lap_{\Z})^\sigma$ and its continuous counterpart
$(-\lap_{\R})^\sigma$ are subcritical if and only if
$0<\sigma<1/2$. In the continuous setting, the classical fractional
Hardy weight for $(-\lap_{\R})^\sigma$ reads
\[
W(x)=4^\sigma
\frac{\Gamma^2\!\left(\frac{1+2\sigma}{4}\right)}
     {\Gamma^2\!\left(\frac{1-2\sigma}{4}\right)}
\frac{1}{|x|^{2\sigma}},
\qquad x\in\R\setminus\{0\},
\]
where the constant is sharp; see~\cite{her_77,yaf_99}.
In the discrete setting, the Hardy weight for $(-\lap_{\Z})^\sigma$ in the form
\[
W(n)=4^\sigma
\frac{\Gamma^2\!\left(\frac{1+2\sigma}{4}\right)}
     {\Gamma^2\!\left(\frac{1-2\sigma}{4}\right)}
\frac{
 \Gamma\!\left(|n|+\frac{1-2\sigma}{4}\right)
 \Gamma\!\left(|n|+\frac{3-2\sigma}{4}\right)}
{\Gamma\!\left(|n|+\frac{1+2\sigma}{4}\right)
 \Gamma\!\left(|n|+\frac{3+2\sigma}{4}\right)},
\qquad n\in\Z
\]
was obtained by Ciaurri and Roncal~\cite{cia-ron_18}. Its optimality for every $\sigma\in(0,1/2)$ was subsequently
proved by Keller and Nietschmann~\cite{kel-nit_23}.

\subsection{Main results}

Recall that we refer to any non-negative sequence $W\neq0$ as a \emph{Hardy weight for} $(-\lap_{\mathrm{D}})^\sigma$ if the inequality
\begin{equation}\label{eq:hardy-ineq}
    \sum_{n=1}^\infty \overline{u_n} (-\lap_{\mathrm{D}})^\sigma u_n \geq \sum_{n=1}^\infty W_n\abs{u_n}^2
\end{equation}
holds for all $u\in\ell^{2}(\N)$. The precise definition of the operator $(-\lap_{\mathrm{D}})^\sigma$ is postponed to Subsection~\ref{subsec:laps}. Before stating our main results, we formulate the exact meaning of the \emph{optimality} used here.
\begin{definition}
    \label{def:optimality}
    A Hardy weight $W$ for $(-\lap_{\mathrm{D}})^\sigma$ is said to be \textit{optimal in} $\ell^{2}(\N)$ if it has the following three properties:
    \begin{enumerate}[label=\normalfont(\roman*)]
        \item
            \textit{Criticality}: $W$ is called \textit{critical} if every Hardy weight $\widetilde W$ for $(-\lap_{\mathrm{D}})^\sigma$ satisfying $\widetilde W_n\geq W_n$ for all $n\in\N$ coincides with $W$.
        \item
            \textit{Non-attainability in }$\ell^{2}(\N)$: $W$ is called \textit{non-attainable in} $\ell^{2}(\N)$ if the inequality in~\eqref{eq:hardy-ineq} is strict for all non-zero $u\in\ell^2(\N)$.
        \item
            \textit{Optimality near infinity}: The weight $W$ is called \textit{optimal near infinity} if for any $M\in\N$ and $\varepsilon>0$ there exists $0\neq u\in \ell^{2}(\N)$ such that $u_n=0$ for all $n<M$ and
        \begin{equation*}
            \sum_{n=1}^\infty\overline{u_n}(-\lap_{\mathrm{D}})^\sigma u_n < (1+\varepsilon) \sum_{n=1}^\infty W_n\abs{u_n}^2.
        \end{equation*}
    \end{enumerate}
\end{definition}

Our first main result is as follows. For $0<\sigma<\alpha<3/2$, we put
\begin{equation}
W_{\alpha,\sigma}(n):=4^\sigma\frac{\Gamma(\frac32-\alpha+\sigma)\Gamma(\alpha)}
{\Gamma(\alpha-\sigma)\Gamma(\frac32-\alpha)}\frac{\Gamma(n+\alpha-\sigma-1)\Gamma(n-\alpha+2)}{\Gamma(n-\alpha+\sigma+2)\Gamma(n+\alpha-1)}.
\label{eq:def_W}
\end{equation}

\begin{theorem}\label{thm:main}
Let $1<\sigma<\alpha<3/2$. Then the following claims hold.
\begin{enumerate}[label=\normalfont(\roman*), itemsep=2pt]
\item $W_{\alpha,\sigma}$ is a Hardy weight for $(-\lap_{\mathrm{D}})^\sigma$.
\item $W_{\alpha,\sigma}$ is critical if and only if $\alpha\leq \alpha_\sigma^*:=(3+2\sigma)/4$.
\item $W_{\alpha,\sigma}$ is non-attainable in $\ell^{2}(\N)$.
\item For $\alpha=\alpha_\sigma^*$, the Hardy weight
\begin{equation}\label{eq:wop_explicit}
W_\sigma^{\mathrm{op}}(n):=W_{\alpha_\sigma^*,\sigma}(n)=4^\sigma
\frac{\Gamma^{2}(\frac{3+2\sigma}{4})}{\Gamma^{2}(\frac{3-2\sigma}{4})}
\frac{\Gamma(n-\frac{1+2\sigma}{4})\Gamma(n+\frac{5-2\sigma}{4})}{\Gamma(n-\frac{1-2\sigma}{4})\Gamma(n+\frac{5+2\sigma}{4})}
\end{equation}
is optimal in $\ell^{2}(\N)$.
\end{enumerate}
\end{theorem}

\begin{remark}
Theorem~\ref{thm:main} complements
\cite[Theorem~1.1]{das-fer_26}, which treats
$0<\sigma\leq1$ using a different notion of optimality;
see the next remark. Since their notion implies optimality
in the sense of Definition~\ref{def:optimality}, the two
results together show that $W_\sigma^{\mathrm{op}}$ is optimal
in $\ell^2(\N)$ for every $0<\sigma<3/2$, covering the full
subcriticality range of $(-\lap_{\mathrm{D}})^\sigma$. 
After the publication of~\cite{das-fer_26}, the authors added to the revised arXiv version~\cite{das-fer_26_arxiv_v2} the inequality 
\begin{equation}
    W_\sigma^{\mathrm{op}}(n) \ge \frac{\gamma_\sigma}{n^{2\sigma}},
    \quad\text{ with }\; \gamma_\sigma:= 4^\sigma\frac{\Gamma^{2}(\frac{3+2\sigma}{4})}{\Gamma^{2}(\frac{3-2\sigma}{4})};
\label{eq:W-lower-bound}
\end{equation}
see \cite[Lemma~5.1]{das-fer_26_arxiv_v2}, where it is stated for $0<\sigma<1$. In fact, the inequality holds throughout $0<\sigma<3/2$; we provide an independent proof in
Appendix~\ref{app:weight_lower_bound}. Consequently,
\begin{equation}
(-\lap_\mathrm{D})^\sigma\ge \frac{\gamma_\sigma}{n^{2\sigma}}
\label{eq:ineq-clas-hardy-type}
\end{equation}
on $\ell^{2}(\N)$ for all $\sigma\in(0,3/2)$. Moreover, by the asymptotic behaviour $W_\sigma^{\mathrm{op}}(n) \sim \gamma_\sigma n^{-2\sigma}$ for $n\to\infty$, which stems from expansion~\eqref{eq:Gamma_ratio_asymp} below, and optimality of $W_\sigma^{\mathrm{op}}$ near infinity, the constant $\gamma_\sigma$ in \eqref{eq:ineq-clas-hardy-type} is sharp for all $\sigma\in(0,3/2)$.
\end{remark}

\begin{remark}
The notion of optimality used in \cite{das-fer_26} requires
criticality and null-criticality, and implies optimality
near infinity. For a critical Hardy weight $W$,
null-criticality means that the positive Agmon ground state
$g$ of $(-\lap_{\mathrm{D}})^\sigma-W$ does not belong to the weighted space $\ell^{2}(\N,W)$; see \cite[Definition~2.9]{das-fer_26}.
In the graph setting, this implies non-attainability in
$\ell^2(\N)$: equality cases are multiples of $g$, and
$\ell^2(\N)\subseteq\ell^2(\N,W)$, since every Hardy
weight for the bounded operator $(-\lap_{\mathrm{D}})^\sigma$
is bounded. Thus null-criticality excludes the ground state
from a larger space than $\ell^2(\N)$. However, this formulation relies on results from
the graph Schr{\"o}dinger operator theory~\cite{kel-pin-pog_20}, which does not
apply directly when $\sigma>1$, as mentioned above. We therefore formulate non-attainability
directly in $\ell^2(\N)$ and establish optimality near
infinity separately.
\end{remark}

\begin{remark}
Without going into details, we remark that the notion of non-attainability can be strengthened.
Namely, one may consider the homogeneous energy space $\mathcal{H}_\sigma$, defined as the completion
of $\ell^2(\N)$ in the norm
\[
    \|u\|_\sigma:=\|(-\lap_{\mathrm{D}})^{\sigma/2}u\|_{\ell^2(\N)}.
\]
For the weights in Theorem~\ref{thm:main}, the fractional Hardy
inequality yields a continuous embedding $\mathcal{H}_\sigma\hookrightarrow \ell^2(\N,W_{\alpha,\sigma})$ and, for $\alpha\geq\alpha_\sigma^*$, non-attainability strengthens to the strict inequality
\[
    \|u\|_\sigma^2
    >
    \sum_{n=1}^\infty W_{\alpha,\sigma}(n)|u_n|^2
\]
for all $0\neq u\in\mathcal{H}_\sigma$. Nevertheless, we retain the $\ell^2(\N)$ formulation in our main
results.
\end{remark}

The weights in~\eqref{eq:def_W}, including the optimal weight~\eqref{eq:wop_explicit}, are constructed using a Riesz sequence associated with the fractional Dirichlet Laplacian (see~\eqref{eq:def_I_alp}). Alternatively, replacing the Riesz sequence with a parameter sequence (see \eqref{eq:def_g_sigma}) motivated by our previous work~\cite{sta-wac_26} yields a new optimal Hardy weight distinct from~\eqref{eq:wop_explicit} for all exponents $\sigma \in (0, 3/2)$, except when $\sigma=1/2$.

\begin{theorem}
    \label{thm:new_opt_weight}
    Let $0<\sigma<3/2$. Then the sequence
    \begin{equation}
        \label{eq:new_opt_weight}
        \widehat{W}_{\sigma}^{\mathrm{op}}(n)
          := \frac{\Gamma^2(\sigma+1/2)}{\pi} \left(4^\sigma\frac{\Gamma(2n-1)}{\Gamma(2n-1+2\sigma)}+\sin(\pi\sigma)\frac{\Gamma(n)\Gamma(n-\sigma+1)}{\Gamma(n+3/2)\Gamma(n+\sigma-1/2)}\right)
    \end{equation}
    is an optimal Hardy weight for $(-\lap_\mathrm{D})^\sigma$ in $\ell^2(\N)$.
\end{theorem}

\begin{remark}
    The optimal weight satisfies $\widehat{W}_{\sigma}^{\mathrm{op}}(n)>0$ for all $n\in\N$ and $\sigma\in(0,3/2)$.
\end{remark}

\begin{remark}
     For $\sigma=1$, $(-\lap_{\rm D})^{\sigma}$ coincides with the Toeplitz Laplacian $-\lap$ from \cite{sta-wac_26} and the optimal weight \eqref{eq:new_opt_weight} coincides with \eqref{eq:sta-wac_opt_weight}. Moreover, for the square root $(-\lap_{\rm D})^{1/2}$, the two optimal weights coincide:
    \[
    W_{1/2}^{\mathrm{op}}(n)=\widehat W_{1/2}^{\mathrm{op}}(n)=\frac{2n}{\pi(n^2-1/4)}
    \]
    for all $n\in\N$.
\end{remark}

\subsection{Organisation}

Section~\ref{sec:pre} contains the necessary preliminaries.
In particular, Proposition~\ref{prop:two-term-repre} provides
a two-term ground state representation, which may be of
independent interest. We use this representation in
Sections~\ref{sec:proofs} and~\ref{sec:new_weight} to prove Theorems~\ref{thm:main} and~\ref{thm:new_opt_weight}, respectively. Appendix~\ref{app:weight_lower_bound} proves the lower bound \eqref{eq:W-lower-bound}.

\section{Preliminaries}\label{sec:pre}

\subsection{Notation}
We use the notation $\Z$ and $\N$ for the sets of integers and positive integers, respectively. We denote by $C(\N)$ the space of complex sequences indexed by $\N$ and by $\ell^{2}(\N)$ and $C_{0}(\N)$ its subspaces of square-summable and finitely supported sequences, respectively.

\subsection{Discrete Dirichlet and Neumann Laplacians}\label{subsec:laps}

On $\ell^{2}(\N)$, we define the \emph{discrete Dirichlet} and \emph{Neumann Laplacians} by
\[
  -\lap_{\mathrm{D}}:=\grad_{\mathrm{D}}^{*}\grad_{\mathrm{D}}
  \quad\mbox{ and }\quad
  -\lap_{\mathrm{N}}:=\grad_{\mathrm{D}}\grad_{\mathrm{D}}^{*},
\]
where the \emph{discrete Dirichlet gradient} $\grad_{\mathrm{D}}$ is determined by
\[
 \grad_{\mathrm{D}}u_n :=\begin{cases}
                    u_1 & \mbox{ if } n=1, \\
                    u_n-u_{n-1} & \mbox{ if } n\geq2,
                    \end{cases}
\]
for $u\in\ell^{2}(\N)$, and $\grad_{\mathrm{D}}^{*}$ is its adjoint acting by $\grad_{\mathrm{D}}^{*}u_n:=u_{n}-u_{n+1}$ for all $n\in\N$. Clearly, by their actions, the discrete difference operators $\grad_{\mathrm{D}}$ and $\grad_{\mathrm{D}}^{*}$ are well defined on the space of all complex sequences $C(\N)$, not necessarily from $\ell^{2}(\N)$. Similarly, the actions of $-\lap_{\mathrm{D}}$ and $-\lap_{\mathrm{N}}$, which read
\[
 (-\lap_{\mathrm{D}}u)_n=
 \begin{cases}
   2u_1-u_2, & n=1,\\
   -u_{n-1}+2u_n-u_{n+1}, & n\geq2,
 \end{cases}
\]
and
\[
 (-\lap_{\mathrm{N}}u)_n=
 \begin{cases}
   u_1-u_2, & n=1,\\
   -u_{n-1}+2u_n-u_{n+1}, & n\geq2,
 \end{cases}
\]
make sense for all $u\in C(\N)$, and in this sense, their definitions are understood whenever they apply to sequences outside of $\ell^{2}(\N)$.

As operators on $\ell^{2}(\N)$, $-\lap_{\mathrm{D}}$ and $-\lap_{\mathrm{N}}$ are self-adjoint, both with purely absolutely continuous spectrum filling the interval $[0,4]$. In addition, a direct calculation shows that
\begin{equation}
 S(-\lap_{\mathrm{D}})S^{*}=C(-\lap_{\mathrm{N}})C^{*}=M_{\lambda},
\label{eq:laps_SC_M}
\end{equation}
where $M_{\lambda}$ is the multiplication operator by the function
\[
\lambda(\theta):=4\sin^2(\theta/2)
\]
on $L^2(0,\pi)$, and $S,C:\ell^2(\N)\to L^2(0,\pi)$ are the unitary sine and cosine transforms determined by
\[
Su(\theta):=\sqrt{\frac2\pi}\sum_{n=1}^{\infty}u_{n}\sin(n\theta)
\quad\mbox{ and }\quad
Cu(\theta):=\sqrt{\frac2\pi}\sum_{n=1}^{\infty}u_{n}\cos\bigl((n-\tfrac12)\theta\bigr)
\]
for $u\in\ell^2(\N)$, with convergence in $L^2(0,\pi)$. Consequently, the fractional discrete Laplacians are defined by the functional calculus as
\begin{equation*}
(-\lap_{\mathrm{D}})^{\sigma}:=S^{*}M_{\lambda^\sigma}S
\quad\mbox{ and }\quad
(-\lap_{\mathrm{N}})^{\sigma}:=C^{*}M_{\lambda^\sigma}C
\end{equation*}
for any exponent $\sigma>0$. These two fractional Laplacians are related, a property we shall use below.

\begin{lemma}\label{lem:factorisation}
    On $\ell^{2}(\N)$, we have, for all $\sigma>0$, the factorisation formula
    \begin{equation}
        (-\lap_{\mathrm{D}})^{\sigma+1}=\grad_{\mathrm{D}}^{*}\,(-\lap_{\mathrm{N}})^{\sigma}\,\grad_{\mathrm{D}}
        \label{eq:factorisation}
    \end{equation}
    and the intertwining identity
    \begin{equation}
    (-\lap_{\mathrm{N}})^\sigma\,\grad_{\mathrm{D}}
    =\grad_{\mathrm{D}}\,(-\lap_{\mathrm{D}})^\sigma.
    \label{eq:intertwining}
\end{equation}
\end{lemma}

\begin{proof}
    First, using the definitions together with the elementary identity
    \[
    \cos\bigl((n-\tfrac12)\theta\bigr)-\cos\bigl((n+\tfrac12)\theta\bigr)=2\sin(\theta/2)\sin(n\theta),
    \]
    one readily verifies that $C\,\grad_{\mathrm{D}}=M_{\sqrt{\lambda}}S$, and also
    $\grad_{\mathrm{D}}^{*}C^{*}=S^{*}M_{\sqrt{\lambda}}$ by taking adjoints.
    Using the last two identities, we find
    \[
    (-\lap_{\mathrm{D}})^{\sigma+1}=S^{*}M_{\lambda^{\sigma+1}}S=S^{*}M_{\sqrt{\lambda}}M_{\lambda^{\sigma}}M_{\sqrt{\lambda}}S
    =\grad_{\mathrm{D}}^{*}\,C^{*}M_{\lambda^{\sigma}}C\,\grad_{\mathrm{D}}=\grad_{\mathrm{D}}^{*}\,(-\lap_{\mathrm{N}})^{\sigma}\,\grad_{\mathrm{D}},
    \]
    which is \eqref{eq:factorisation}. Similarly,
    \[
    (-\lap_{\mathrm{N}})^\sigma\grad_{\mathrm{D}}=C^*M_{\lambda^{\sigma+1/2}}S=\grad_{\mathrm{D}}(-\lap_{\mathrm{D}})^\sigma,
    \]
    which proves~\eqref{eq:intertwining}.
\end{proof}

The matrix representation of $(-\lap_{\mathrm{D}})^{\sigma}$ with respect to the standard basis of $\ell^{2}(\N)$ was computed in \cite[Proposition~2.1]{ger-kre-sta_25}; the matrix entries read
\begin{equation}
    (-\lap_{\mathrm{D}})^{\sigma}_{m,n}=\frac{(-1)^{m-n}\Gamma(2\sigma+1)}{\Gamma(\sigma+1+m-n)\Gamma(\sigma+1-m+n)}+\frac{(-1)^{m+n+1}\Gamma(2\sigma+1)}{\Gamma(\sigma+1+m+n)\Gamma(\sigma+1-m-n)}
    \label{eq:lap-dir_entries}
\end{equation}
for all $m,n\in\N$ and $\sigma>0$. A completely analogous computation, applying the Fourier integral
\[
\frac{1}{\pi}\int_{0}^{\pi}\lambda^{\sigma}(\theta)\cos(k\theta)\dd\theta = \frac{(-1)^{k}\Gamma(2\sigma+1)}{\Gamma(\sigma+k+1)\Gamma(\sigma-k+1)}, \quad k\in\Z,
\]
which can be deduced from the beta integral identity~\cite[Eq.~3.631(9)]{grad-ryz_94}, in the expression
\begin{align*}
    (-\lap_{\mathrm{N}})^{\sigma}_{m,n}&=\frac{2}{\pi}\int_{0}^{\pi}\lambda^{\sigma}(\theta)\cos((m-1/2)\theta)\cos((n-1/2)\theta)\dd\theta \\
    &=\frac{1}{\pi}\int_{0}^{\pi}\lambda^{\sigma}(\theta)\cos((m-n)\theta)\dd\theta+\frac{1}{\pi}\int_{0}^{\pi}\lambda^{\sigma}(\theta)\cos((m+n-1)\theta)\dd\theta
\end{align*}
gives the Neumann Laplacian matrix entries
\begin{equation}
(-\lap_{\mathrm{N}})^{\sigma}_{m,n}=\frac{(-1)^{m-n}\Gamma(2\sigma+1)}{\Gamma(\sigma+1+m-n)\Gamma(\sigma+1-m+n)}+\frac{(-1)^{m+n+1}\Gamma(2\sigma+1)}{\Gamma(\sigma+m+n)\Gamma(\sigma+2-m-n)}
    \label{eq:lap-neu_entries}
\end{equation}
for all $m,n\in\N$ and $\sigma>0$. Since the reciprocal Gamma function is an entire function vanishing at the non-positive integers, the matrices of $(-\lap_{\mathrm{D}})^{\sigma}$ and $(-\lap_{\mathrm{N}})^{\sigma}$ are banded for $\sigma\in\N$.

It is important to notice that both terms in~\eqref{eq:lap-neu_entries} are strictly negative for all $0<\sigma<1$ and $m\neq n$. Indeed, the Euler reflection formula \cite[Eq.~(5.5.3)]{dlmf}
\begin{equation}\label{eq:euler_reflection}
        \Gamma(z)\Gamma(1-z)=\frac{\pi}{\sin(\pi z)}
\end{equation}
gives
\[
\frac{(-1)^k\Gamma(2\sigma+1)}
{\Gamma(\sigma+k+1)\Gamma(\sigma-k+1)}
=
-\frac{\Gamma(2\sigma+1)\sin(\pi \sigma)}{\pi}
\frac{\Gamma(k-\sigma)}{\Gamma(k+\sigma+1)}
<0
\]
for every $k\in\N$. The two terms in~\eqref{eq:lap-neu_entries} correspond to $k=|m-n|\in\N$ and $k=m+n-1\in\N$, respectively. Consequently,
\begin{equation}
(-\lap_{\mathrm{N}})^\sigma_{m,n}<0
\quad\mbox{ for all }m\neq n\text{ and }0<\sigma<1.
\label{eq:lap_N_neg}
\end{equation}

Both fractional Laplacians extend to larger sequence spaces by their matrix representations. We denote their domains of absolute convergence by
\[
\mathcal{D}_{\mathrm{X}}^{(\sigma)}:=\Bigl\{ u\in C(\N) \Bigm\vert \sum_{n=1}^\infty \bigl|(-\lap_{\mathrm{X}})^{\sigma}_{m,n} u_n\bigr| < \infty,\;\forall m\in\N \Bigr\} \quad\mbox{ for } \mathrm{X}\in\{\mathrm{D},\mathrm{N}\}.
\]
Using~\eqref{eq:euler_reflection} in conjunction with the Gamma ratio asymptotics~\cite[Eq.~(5.11.13)]{dlmf}
\begin{equation}
    \label{eq:Gamma_ratio_asymp}
    \frac{\Gamma(x+a)}{\Gamma(x+b)}=x^{a-b}\left[1+\frac{(a-b)(a+b-1)}{2x}+\mathcal{O}\!\left(\frac{1}{x^{2}}\right)\!\right]
    \quad\text{ as } x\to\infty,
\end{equation}
in \eqref{eq:lap-dir_entries} and \eqref{eq:lap-neu_entries}, one infers, as $n\to\infty$, that
\begin{equation}
(-\lap_{\mathrm D})^\sigma_{m,n}=-\frac{2\Gamma(2\sigma+2)\sin(\pi\sigma)}{\pi}
\frac{m}{n^{2\sigma+2}}\left[1+\mathcal{O}\!\left(\frac{1}{n}\right)\right]
\label{eq:asympt_lap_D}
\end{equation}
and 
\begin{equation}
(-\lap_{\mathrm N})^\sigma_{m,n}=
-\frac{2\Gamma(2\sigma+1)\sin(\pi\sigma)}{\pi}
\frac{1}{n^{2\sigma+1}}\left[1+\mathcal{O}\!\left(\frac{1}{n}\right)\right]
\label{eq:asympt_lap_N}
\end{equation}
for fixed $m\in\mathbb N$ and non-integer $\sigma>0$.

Next we show that the factorisation~\eqref{eq:factorisation} and the intertwining identity \eqref{eq:intertwining}, although stated on $\ell^2(\N)$, remain valid pointwise when both sides are applied to any sequence $g$ with $\grad_{\rm D}g\in\mathcal{D}_{\rm N}^{(\sigma)}$.

\begin{lemma}\label{lem:factorisation_extended}
    Let $\sigma>0$ and $g\in C(\N)$ be such that
    $\grad_{\rm D}g\in\mathcal{D}_{\rm N}^{(\sigma)}$.
    Then $g\in\mathcal{D}_{\rm D}^{(\sigma)}\cap
    \mathcal{D}_{\rm D}^{(\sigma+1)}$ and the intertwining identity
    \begin{equation}\label{eq:intertwining_extended}
        (-\lap_{\rm N})^\sigma\grad_{\rm D}g
        =\grad_{\rm D}(-\lap_{\rm D})^\sigma g
    \end{equation}
    as well as the factorisation formula
    \begin{equation}\label{eq:factorisation_extended}
        (-\lap_{\rm D})^{\sigma+1}g
        =\grad_{\rm D}^*(-\lap_{\rm N})^\sigma\grad_{\rm D}g
    \end{equation}
    hold pointwise.
\end{lemma}

\begin{proof}
    Without loss of generality, we may assume that $\sigma\notin\N$.
    Indeed, for integer $\sigma$ all the operators appearing in
    \eqref{eq:intertwining_extended} and
    \eqref{eq:factorisation_extended} are represented by banded
    matrices, and both identities follow from
    Lemma~\ref{lem:factorisation} for any $g\in C(\N)$.

    We first establish a summability property of $g$. By \eqref{eq:asympt_lap_N}, the assumption $\grad_{\rm D}g\in\mathcal{D}_{\rm N}^{(\sigma)}$ implies
    \begin{equation}\label{eq:aux_grad_summability}
        \sum_{n=1}^\infty \frac{|\grad_{\rm D}g_n|}{n^{2\sigma+1}}<\infty.
    \end{equation}
    Since $g_n=\sum_{k=1}^n\grad_{\rm D}g_k$, changing the order of summation gives
    \begin{equation}
        \label{eq:aux_g_summability}
      \sum_{n=1}^\infty \frac{|g_n|}{n^{2\sigma+2}}\leq\sum_{k=1}^\infty|\grad_{\rm D}g_k|\sum_{n=k}^\infty\frac{1}{{n^{2\sigma+2}}}
        \leq C_\sigma\sum_{k=1}^\infty\frac{|\grad_{\rm D}g_k|}{k^{2\sigma+1}}<\infty,
    \end{equation}
    where $C_\sigma>0$ depends only on $\sigma$. Moreover, \eqref{eq:aux_grad_summability} implies
    \begin{equation}\label{eq:aux_g_growth}
        g_n=o(n^{2\sigma+1})\quad\text{as }n\to\infty.
    \end{equation}
    Indeed, for $m<n$ we have
    \[
        \frac{|g_n|}{n^{2\sigma+1}} \leq \frac{|g_m|}{n^{2\sigma+1}}+\sum_{k=m+1}^n\frac{|\grad_{\rm D}g_k|}{n^{2\sigma+1}}
        \leq \frac{|g_m|}{n^{2\sigma+1}}+\sum_{k=m+1}^{n}\frac{|\grad_{\rm D}g_k|}{k^{2\sigma+1}}.
    \]
    Sending first $n\to\infty$ and then $m\to\infty$ implies \eqref{eq:aux_g_growth}. Using \eqref{eq:asympt_lap_D} together with \eqref{eq:aux_g_summability}, one shows that $g\in\mathcal{D}_{\rm D}^{(\sigma)}\cap
    \mathcal{D}_{\rm D}^{(\sigma+1)}$.

    Next we prove \eqref{eq:intertwining_extended}. By convention, we put $(-\lap_{\rm D})^\sigma_{0,n}:=0$ for all $n\in\N$.
    Applying \eqref{eq:intertwining} to the standard basis vectors of $\ell^2(\N)$ gives
    \[
        (-\lap_{\rm N})^\sigma_{m,n}-(-\lap_{\rm N})^\sigma_{m,n+1}=(-\lap_{\rm D})^\sigma_{m,n}-(-\lap_{\rm D})^\sigma_{m-1,n}
    \]
    for all $m,n\in\N$. Hence, for any $m,N\in\N$, summation by parts yields
    \begin{align}
        \sum_{n=1}^N (-\lap_{\rm N})^\sigma_{m,n}\grad_{\rm D}g_n
        &=(-\lap_{\rm N})^\sigma_{m,N}g_N + \sum_{n=1}^{N-1}\left((-\lap_{\rm N})^\sigma_{m,n}-(-\lap_{\rm N})^\sigma_{m,n+1}\right)g_n \nonumber\\
        &=(-\lap_{\rm N})^\sigma_{m,N}g_N + \sum_{n=1}^{N-1}\left((-\lap_{\rm D})^\sigma_{m,n}-(-\lap_{\rm D})^\sigma_{m-1,n}\right)g_n.
        \label{eq:aux_parts_intertwining}
    \end{align}
    The sum on the left-hand side converges absolutely as $N\to\infty$ by the assumption on $\grad_{\rm D}g$.
    The sum on the right-hand side also converges absolutely since $g\in\mathcal{D}_{\rm D}^{(\sigma)}$.
    Moreover, \eqref{eq:asympt_lap_N} and \eqref{eq:aux_g_growth} imply $(-\lap_{\rm N})^\sigma_{m,N}g_N\to0$ as $N\to\infty$.
    Passing to the limit $N\to\infty$ in \eqref{eq:aux_parts_intertwining}, we conclude, for all $m\in\N$, that
    \[
        (-\lap_{\rm N})^\sigma\grad_{\rm D}g_m =(-\lap_{\rm D})^\sigma g_m - (-\lap_{\rm D})^\sigma g_{m-1},
    \]
    where the second term on the right is understood as zero for $m=1$. This proves \eqref{eq:intertwining_extended}.

    To prove \eqref{eq:factorisation_extended}, we temporarily denote $D_{m,n}:=(-\lap_{\rm N})^\sigma_{m,n}-(-\lap_{\rm N})^\sigma_{m+1,n}$ for $m,n\in\N$.  Applying \eqref{eq:factorisation} to the standard basis vectors of $\ell^2(\N)$ yields
    \[
    (-\lap_{\rm D})^{\sigma+1}_{m,n} =D_{m,n}-D_{m,n+1}
    \]
    for all $m,n\in\N$. As above, summation by parts gives
    \begin{equation}
        \sum_{n=1}^N D_{m,n}\grad_{\rm D}g_n=D_{m,N}g_N + \sum_{n=1}^{N-1}(D_{m,n}-D_{m,n+1})g_n
        =D_{m,N}g_N + \sum_{n=1}^{N-1}(-\lap_{\rm D})^{\sigma+1}_{m,n}g_n.
        \label{eq:aux_parts_factorisation}
    \end{equation}
    The sum on the left-hand side converges absolutely by the assumption on $\grad_{\rm D}g$. The sum on the right-hand side converges absolutely since $g\in\mathcal{D}_{\rm D}^{(\sigma+1)}$. Finally, \eqref{eq:asympt_lap_N} together with \eqref{eq:aux_g_growth} implies $D_{m,N}g_N\to0$ as $N\to\infty$ for any $m\in\N$ fixed. Passing to the limit $N\to\infty$ in \eqref{eq:aux_parts_factorisation} proves \eqref{eq:factorisation_extended} pointwise.
\end{proof}

\subsection{A two-term ground state representation}

The key idea in deriving the Hardy weights for Theorems \ref{thm:main} and \ref{thm:new_opt_weight} relies on establishing a suitable ground state representation formula~\cite{fra-sei_08}. We will combine two ground state formulas. First, we recall the following local weighted ground state transform, which has been proven in several papers; see, for example, \cite[Eq.~(1.10)]{huang-ye_24} or \cite[Theorem~0]{sta-wac_24}.

\begin{lemma}\label{lem:weighted_hardy}
    Let $V,g\in C(\N)$ and $g_{n}>0$ for all $n\in\N$. Then for all $u\in C_0(\N)$, we have the identity
    \begin{equation}
        \label{eq:weighted_hardy}
        \sum_{n=1}^\infty  V_{n} \abs{\grad_{\mathrm{D}} u_{n}}^{2} - \sum_{n=1}^\infty  \frac{\grad_{\mathrm{D}}^{*}V\grad_{\mathrm{D}}\, g_{n}}{g_n}\abs{u_{n}}^{2}
          = \sum_{n=1}^\infty  V_{n+1}g_{n+1}g_{n}\abs{\frac{u_{n+1}}{g_{n+1}}-\frac{u_{n}}{g_{n}}}^2.
    \end{equation}
\end{lemma}

The next lemma gives an elementary ground state transform for the fractional Laplacians $(-\lap_{\mathrm{X}})^{\sigma}$, $\mathrm{X}\in\{\mathrm{D},\mathrm{N}\}$, analogous to results proven in \cite[Proposition~3.4]{das-fer_26} and \cite[Theorem~3.2]{sta-wac_26}. Its proof is straightforward, but is included for convenience.

\begin{lemma}\label{lem:frac_lap_N_id}
    Let $\mathrm{X}\in\{\mathrm{D},\mathrm{N}\}$, $\sigma>0$, and $h\in\mathcal{D}_{\mathrm{X}}^{(\sigma)}$ be such that $h_n>0$ for all $n\in\N$. Then for all $v\in C_0(\N)$, we have     \begin{equation}
        \label{eq:frac_lap_N_id}
        \sum_{n=1}^\infty  \overline{v_n} (-\lap_{\mathrm{X}})^\sigma v_n - \sum_{n=1}^\infty \frac{(-\lap_{\mathrm{X}})^\sigma h_n}{h_n} \abs{v_n}^2
          = -\!\!\sum_{1\leq n<m} \!(-\lap_{\mathrm{X}})^\sigma_{m,n}\,h_m\,h_n\abs{\frac{v_m}{h_m}-\frac{v_n}{h_n}}^2.
    \end{equation}
\end{lemma}
\begin{proof}
    Multiplying both sides of the equality
    \[
        h_m h_n\abs{\frac{v_m}{h_m}-\frac{v_n}{h_n}}^2=\frac{h_n}{h_m}\abs{v_m}^2+\frac{h_m}{h_n}\abs{v_n}^2 -\overline{v_m}v_n-\overline{v_n}v_m
    \]
    by $(-\lap_{\mathrm{X}})^\sigma_{m,n}$, summing over all $m,n\in\N$, and using the symmetry of the expressions, we get
    \[
        \sum_{m,n=1}^{\infty} (-\lap_{\mathrm{X}})^\sigma_{m,n}\,h_m\,h_n\abs{\frac{v_m}{h_m}-\frac{v_n}{h_n}}^2=
        2\sum_{n=1}^\infty \frac{(-\lap_{\mathrm{X}})^\sigma h_n}{h_n} \abs{v_n}^2-2\sum_{n=1}^\infty  \overline{v_n} (-\lap_{\mathrm{X}})^\sigma v_n,
    \]
    which implies the claim.
\end{proof}

If $\sigma\in(0,1]$, the off-diagonal matrix entries of $(-\lap_\mathrm{D})^\sigma$ are non-positive (see \cite[Proposition~2.2]{das-fer_26}), and therefore the remainder term on the right-hand side of~\eqref{eq:frac_lap_N_id}, with $\mathrm{X}=\mathrm{D}$, is non-negative. If, on the other hand, $\sigma\in(1,3/2)$, the summands vary in sign. In order to show that the remainder is non-negative even in this case, we combine identities \eqref{eq:frac_lap_N_id}, with $\mathrm{X}=\mathrm{N}$, and \eqref{eq:weighted_hardy} to produce a ground state identity with two terms on the right, which are both non-negative under additional assumptions on the sequence $g$.

\begin{proposition}\label{prop:two-term-repre}
    Let $\sigma\in(1,3/2)$. Suppose $g\in C(\N)$ is such that $\grad_\mathrm{D}g\in\mathcal{D}_\mathrm{N}^{(\sigma-1)}$ with $g_n>0$ and $\grad_\mathrm{D} g_n>0$ for all $n\in\N$. Then for all $u\in C_0(\N)$, we have
    \[
        \sum_{n=1}^\infty \overline{u_n} (-\lap_{\mathrm{D}})^\sigma u_n - \sum_{n=1}^\infty \frac{(-\lap_\mathrm{D})^\sigma g_n}{g_n}\abs{u_n}^2 = R_{\mathrm{N}}[g,u]+R_{\loc}[g,u],
    \]
    where 
    \[
        R_{\mathrm{N}}[g,u]:=-\!\!\sum_{1\leq n<m} \!(-\lap_{\mathrm{N}})^{\sigma-1}_{m,n}\,\grad_{\mathrm{D}}g_m\grad_{\mathrm{D}}g_n\abs{\frac{\grad_{\mathrm{D}}u_m}{\grad_{\mathrm{D}}g_m}-\frac{\grad_{\mathrm{D}}u_n}{\grad_{\mathrm{D}}g_n}}^2
    \]
    and
    \[
        R_{\loc}[g,u]:=\sum_{n=1}^{\infty}\frac{(-\lap_\mathrm{N})^{\sigma-1}\grad_{\mathrm{D}}g_{n+1}}{\grad_{\mathrm{D}}g_{n+1}}\,g_{n+1}\,g_n\,\abs{\frac{u_{n+1}}{g_{n+1}}-\frac{u_{n}}{g_n}}^{2}.
    \]
\end{proposition}

\begin{proof}
For $u\in C_0(\N)$, put $v:=\grad_{\mathrm{D}}u\in C_0(\N)$. The factorisation~\eqref{eq:factorisation} implies
    \[
     \sum_{n=1}^\infty\overline{u_n}(-\lap_{\mathrm{D}})^\sigma u_n=\sum_{n=1}^\infty\overline{v_n}(-\lap_{\mathrm{N}})^{\sigma-1}\, v_n.
    \]
Further, set $V_n:=(-\lap_\mathrm{N})^{\sigma-1}h_n/h_n$, where $h:=\grad_\mathrm{D}g$. Then Lemma~\ref{lem:frac_lap_N_id}, with $\mathrm{X}=\mathrm{N}$, yields
    \begin{equation}
    \sum_{n=1}^\infty\overline{u_n}(-\lap_{\mathrm{D}})^\sigma u_n
     -\sum_{n=1}^\infty V_n|v_n|^2
     =-\sum_{1\leq n<m}(-\lap_{\mathrm{N}})^{\sigma-1}_{m,n}\,h_mh_n
       \abs{\frac{v_m}{h_m}-\frac{v_n}{h_n}}^2.
    \label{eq:gr_rep_1_inproof}
    \end{equation}
On the other hand, we have the pointwise equation  
    \[
    \grad_\mathrm{D}^* V\grad_\mathrm{D}g=\grad_\mathrm{D}^*(-\lap_\mathrm{N})^{\sigma-1}\grad_\mathrm{D}g=(-\lap_\mathrm{D})^\sigma g,
    \]
   where the second equality holds by Lemma~\ref{lem:factorisation_extended}. Consequently, Lemma~\ref{lem:weighted_hardy} gives
    \begin{equation}
    \sum_{n=1}^\infty V_n|v_n|^2 = \sum_{n=1}^\infty \frac{(-\lap_\mathrm{D})^\sigma g_n}{g_n}|u_n|^2
     +\sum_{n=1}^\infty V_{n+1}g_{n+1}g_{n}\abs{\frac{u_{n+1}}{g_{n+1}}-\frac{u_{n}}{g_{n}}}^2.
    \label{eq:gr_rep_2_inproof}
    \end{equation}
Combining \eqref{eq:gr_rep_1_inproof} and \eqref{eq:gr_rep_2_inproof} and substituting the expressions for $h$, $v$, and $V$ proves the claim. The double sum is absolutely convergent because $v$ is finitely supported and $h\in\mathcal{D}_{\mathrm{N}}^{(\sigma-1)}$; the local sum is finite.
\end{proof}

\begin{remark}
Bearing in mind \eqref{eq:lap_N_neg}, the non-local remainder $R_\mathrm{N}[g,u]$ is obviously non-negative under the assumptions of Proposition~\ref{prop:two-term-repre}. On the other hand, the non-negativity of the local remainder $R_\loc[g,u]$ is guaranteed by the additional assumption $(-\lap_\mathrm{N})^{\sigma-1}\grad_{\mathrm{D}}g_{n}\ge0$ for all $n\ge2$.    
\end{remark}

\subsection{A parameter sequence from the Riesz potential}
\label{subsec:riesz}

Proposition \ref{prop:two-term-repre} requires a~single input: a sequence $g$, which we refer to as the \emph{parameter sequence}. 
We will make two specific choices for the parameter sequence, which will produce our Theorems \ref{thm:main} and \ref{thm:new_opt_weight}. 

The first choice is inspired by~\cite{das-fer_26} and uses the Riesz potential sequence associated with $(-\lap_{\rm D})^{\sigma}$. We adopt the notation from~\cite{das-fer_26} and denote, for $n\in\N$ and $\alpha\in(0,3/2)$, the Riesz sequence by
\begin{equation}
    \mathcal{I}_{\alpha}(n):=\sqrt{\frac2\pi}\int_{0}^{\pi}\frac{\sin\theta}{\lambda^{\alpha}(\theta)}\sin(n\theta)\dd\theta
    =C(\alpha)\,\frac{n\Gamma(n+\alpha-1)}{\Gamma(n-\alpha+2)},
    \label{eq:def_I_alp}
\end{equation}
where $C(\alpha)=2^{3/2-2\alpha}\Gamma(3/2-\alpha)/\Gamma(\alpha)$; see \cite[Eqs.~(3.1) and (3.2)]{das-fer_26}. For this range of $\alpha$, the integral in \eqref{eq:def_I_alp} converges and
\begin{equation}
 \mathcal{I}_{\alpha}(n)>0 \quad\mbox{ for all } n\in\N.
 \label{eq:I_alp_pos}
\end{equation}

Further properties of $\mathcal{I}_\alpha$, which we will use below,  are collected in the next lemma. Claim~(ii) was proved in
\cite[Proposition~3.2]{das-fer_26} under the additional restriction $\alpha<1+\sigma$.

\begin{lemma}\label{lem:prop_I_alp}
    Let $\alpha\in(0,3/2)$ and $\sigma>0$. Then we have:
    \begin{enumerate}[label=\normalfont(\roman*), itemsep=2pt]
    \item $\mathcal{I}_{\alpha}\in\mathcal{D}_{\mathrm{D}}^{(\sigma)}$ and $\grad_{\mathrm{D}}\,\mathcal{I}_{\alpha}\in\mathcal{D}_{\mathrm{N}}^{(\sigma)}$.
    \item If $\alpha>\sigma$, then $(-\lap_{\mathrm{D}})^{\sigma}\,\mathcal{I}_{\alpha}=\mathcal{I}_{\alpha-\sigma}$.
    \item If $\alpha>\sigma$, then $(-\lap_{\mathrm{N}})^\sigma\,\grad_{\mathrm{D}}\,\mathcal{I}_\alpha=\grad_{\mathrm{D}}\,\mathcal{I}_{\alpha-\sigma}$.
    \item If $\alpha>\sigma+1$, then $\grad_{\mathrm{D}}^{*}(-\lap_{\mathrm{N}})^{\sigma}\,\grad_{\mathrm{D}}\,\mathcal{I}_{\alpha}=\mathcal{I}_{\alpha-\sigma-1}$.
    \item If $\alpha>1$, then $\grad_{\mathrm{D}}\mathcal{I}_{\alpha}(n)>0$ for all $n\in\N$.
    \end{enumerate}
\end{lemma}

\begin{proof}
    Claim (i): Using~\eqref{eq:Gamma_ratio_asymp} and \eqref{eq:def_I_alp}, one deduces that
    \[
    \mathcal{I}_{\alpha}(n)=\mathcal{O}\!\left(n^{2\alpha-2}\right)
    \quad\mbox{ and }\quad
    \grad_{\mathrm{D}}\,\mathcal{I}_{\alpha}(n)=\mathcal{O}\!\left(n^{2\alpha-3}\right)
    \]
    for $n\to\infty$. Together with \eqref{eq:asympt_lap_D} and \eqref{eq:asympt_lap_N} it follows, for any $m\in\N$, that
    \[
    (-\lap_{\mathrm{D}})^{\sigma}_{m,n}\,\mathcal{I}_{\alpha}(n)=\mathcal{O}\!\left(n^{2(\alpha-\sigma)-4}\right)
    \quad\mbox{ and }\quad
    (-\lap_{\mathrm{N}})^{\sigma}_{m,n}\grad_{\mathrm{D}}\,\mathcal{I}_{\alpha}(n)=\mathcal{O}\!\left(n^{2(\alpha-\sigma)-4}\right)
    \]
    as $n\to\infty$. Consequently, the series
    \[
    \sum_{n=1}^{\infty}(-\lap_{\mathrm{D}})^{\sigma}_{m,n}\,\mathcal{I}_{\alpha}(n)
    \quad\mbox{ and }\quad
    \sum_{n=1}^{\infty}(-\lap_{\mathrm{N}})^{\sigma}_{m,n}\grad_{\mathrm{D}}\,\mathcal{I}_{\alpha}(n)
    \]
    are absolutely convergent whenever $\alpha-\sigma<3/2$, which is guaranteed by the assumptions.

    Claim (ii): For $\varepsilon>0$, we consider the regularisation
    \[
    \mathcal{I}_{\alpha}^\varepsilon(n):=\sqrt{\frac2\pi}\int_{0}^{\pi}\frac{\sin\theta}{(\lambda(\theta)+\varepsilon)^{\alpha}}\sin(n\theta)\dd\theta, \quad n\in\N
    \]
    of~\eqref{eq:def_I_alp}. Since the function $\theta\mapsto\sin\theta/(\lambda(\theta)+\varepsilon)^{\alpha}$ is in $L^{2}(0,\pi)$, its inverse sine transform $\mathcal{I}_{\alpha}^{\varepsilon}\in\ell^{2}(\N)$ for any $\varepsilon>0$. Using the elementary inequality
    $|\sin(n\theta)|\leq n\theta$, we deduce the estimate
    \begin{equation}\label{eq:uniform}
    \abs{\mathcal{I}_\alpha^\varepsilon(n)}\leq K_\alpha n,
    \end{equation}
    for all $\varepsilon>0$ and $n\in\N$, where
    \[
     K_\alpha:=\sqrt{\frac2\pi}\int_0^\pi
    \frac{\theta\sin\theta}{\lambda^\alpha(\theta)}\dd\theta<\infty,
    \]
    since the integrand is $\mathcal{O}(\theta^{2-2\alpha})$ as $\theta\to0_+$.
    Dominated convergence then implies that  $\mathcal{I}_{\alpha}^\varepsilon(n)\to\mathcal{I}_\alpha(n)$ as $\varepsilon\to0_+$ for any $n\in\N$. Using the same bound~\eqref{eq:uniform}, together with~\eqref{eq:asympt_lap_D}, the dominated convergence also gives
    \begin{equation}
        \lim_{\varepsilon\to0_{+}}(-\lap_{\mathrm{D}})^{\sigma}\,\mathcal{I}_{\alpha}^{\varepsilon}(n)=(-\lap_{\mathrm{D}})^{\sigma}\,\mathcal{I}_{\alpha}(n)
        \label{eq:lim_act_laps}
    \end{equation}
    for all $n\in\N$.

    Next, by~\eqref{eq:laps_SC_M}, we have
    \[
    (-\lap_{\mathrm{D}})^{\sigma}\,\mathcal{I}_{\alpha}^\varepsilon(n)=\sqrt{\frac2\pi}\int_0^\pi\frac{\lambda^\sigma(\theta)\sin\theta}{(\lambda(\theta)+\varepsilon)^\alpha}\sin(n\theta)\dd\theta
    \]
    for any $\varepsilon>0$ and $n\in\N$. The absolute value of the integrand, up to an $\varepsilon$-independent multiplicative constant, is bounded by the function $\theta^{2}/\lambda^{\alpha}(\theta)$, which is integrable for $\alpha<3/2$. Therefore, the dominated convergence implies
    \[
    \lim_{\varepsilon\to0_{+}}(-\lap_{\mathrm{D}})^{\sigma}\,\mathcal{I}_{\alpha}^\varepsilon(n)=\sqrt{\frac2\pi}\int_0^\pi
    \frac{\sin\theta}{\lambda^{\alpha-\sigma}(\theta)}\sin(n\theta)\dd\theta=\mathcal{I}_{\alpha-\sigma}(n)
    \]
    for $\alpha>\sigma$. The assumption $\alpha>\sigma$ simply ensures a positive index for the Riesz sequence on the right
    (although the integral defining $\mathcal{I}_\alpha$ also converges for non-positive $\alpha$). Bearing~\eqref{eq:lim_act_laps} in mind, claim (ii) follows.

    Claims (iii) and (iv): By putting $g:=\mathcal{I}_{\alpha}$ in Lemma~\ref{lem:factorisation_extended} and taking into account the already proven claims (i) and (ii), we deduce claims (iii) and (iv) from equations~\eqref{eq:intertwining_extended} and \eqref{eq:factorisation_extended}, respectively.

    Claim (v): From~\eqref{eq:def_I_alp}, one readily computes that
    \begin{equation}
    \label{eq:aux_rec_inproof}
     \grad_{\mathrm{D}}\,\mathcal{I}_{\alpha}(n)=\frac{(2n-1)(\alpha-1)}{(n-1)(n-\alpha+1)}\,\mathcal{I}_{\alpha}(n-1)
    \end{equation}
    for $n\geq2$. Recalling~\eqref{eq:I_alp_pos}, we see that $\grad_{\mathrm{D}}\,\mathcal{I}_{\alpha}(n)>0$ for all $1<\alpha<3/2$ and $n\geq2$. For $n=1$, $\grad_{\mathrm{D}}\mathcal{I}_{\alpha}(1)=\mathcal{I}_{\alpha}(1)>0$.
\end{proof}

\begin{remark}
    For $0<\alpha<1/2$, the sequence $\mathcal{J}_{\alpha}:=\grad_{\mathrm{D}}\,\mathcal{I}_{\alpha+1}$ plays the role of the Riesz potential for the Neumann fractional Laplacian $(-\lap_{\mathrm{N}})^{\alpha}$; one can show that
    \[
    \mathcal{J}_{\alpha}(n)=\sqrt{\frac2\pi}\int_0^\pi\frac{\cos(\theta/2)}{\lambda^{\alpha}(\theta)}\cos\bigl((n-\tfrac12)\theta\bigr)\dd\theta.
    \]
    Moreover, by Lemma~\ref{lem:prop_I_alp}(iii), we have $(-\lap_{\mathrm{N}})^{\sigma}\,\mathcal{J}_{\alpha}=\mathcal{J}_{\alpha-\sigma}$ for $0<\sigma<\alpha<1/2$.
\end{remark}

\subsection{A parameter sequence defined by a Gamma ratio}

In order to deduce the optimal Hardy weight \eqref{eq:new_opt_weight}, we will make use of the sequence 
\begin{equation}\label{eq:def_g_sigma}
    g^{(\sigma)}_n:= \frac{\Gamma(n+\sigma-1/2)}{\Gamma(n)},\qquad n\in\N,
\end{equation}
in place of the Riesz sequence. This specific choice is motivated by our previous work~\cite{sta-wac_26}, where it was used to derive the optimal weights \eqref{eq:sta-wac_opt_weight} for the Toeplitz Laplacian $(-\lap)^\sigma$ for all $\sigma>0$. Auxiliary properties needed in the proof of Theorem~\ref{thm:new_opt_weight} are summarised in the next lemma.

\begin{lemma}\label{lem:lap_g_sigma}
Let $\sigma>0$. Then we have:
\begin{enumerate}[label=\normalfont(\roman*), itemsep=2pt]
\item $g^{(\sigma)}\in\mathcal{D}_{\mathrm{D}}^{(\sigma)}\cap\mathcal{D}_{\mathrm{N}}^{(\sigma)}$.
\item If $\sigma>1$, then $\grad_{\mathrm{D}}g^{(\sigma)}\in\mathcal{D}_{\mathrm{N}}^{(\sigma-1)}$ and $\grad_{\mathrm{D}}g^{(\sigma)}_n>0$ for all $n\in\N$.
\item For all $n\in\N$, 
\begin{equation}\label{eq:lapD_g_sigma}
 (-\lap_{\mathrm{D}})^{\sigma}g^{(\sigma)}_n=\frac{\Gamma^2(\sigma+1/2)}{\pi}\left(\frac{\Gamma(n-1/2)}{\Gamma(n+\sigma)}+\sin(\pi\sigma)\frac{\Gamma(n+1-\sigma)}{\Gamma(n+3/2)}\right).
\end{equation}
\item For all $n\in\N$, 
\begin{equation}\label{eq:lapN_g_sigma}
 (-\lap_{\mathrm{N}})^{\sigma}g^{(\sigma)}_n=\frac{\Gamma^2(\sigma+1/2)}{\pi}\left(\frac{\Gamma(n-1/2)}{\Gamma(n+\sigma)}-\sin(\pi\sigma)\frac{\Gamma(n-\sigma)}{\Gamma(n+1/2)}\right).
\end{equation}
\item If $\sigma\in(0,3/2)$, then $(-\lap_{\mathrm{D}})^{\sigma}g^{(\sigma)}_n>0$ for all $n\in\N$.
\item If $\sigma\in(0,1/2)$, then $(-\lap_{\mathrm{N}})^{\sigma}g^{(\sigma)}_n>0$ for all $n\in\N$.
\end{enumerate}
For $\sigma\in\N$, formulas \eqref{eq:lapD_g_sigma} and \eqref{eq:lapN_g_sigma} are to be interpreted as the respective limits.
\end{lemma}

\begin{proof}
We will assume $\sigma\notin\N$; for integer $\sigma$, claims (i), (ii) are elementary and claims (iii), (iv) extend by continuity. 

Claim (i): By definition~\eqref{eq:def_g_sigma} and asymptotic formula~\eqref{eq:Gamma_ratio_asymp}, we have $g^{(\sigma)}_n=\mathcal{O}(n^{\sigma-1/2})$ for $n\to\infty$. Using also \eqref{eq:asympt_lap_D} and \eqref{eq:asympt_lap_N}, we find, for any $m\in\N$, that 
\[
(-\lap_{\rm D})^{\sigma}_{m,n}g_{n}^{(\sigma)}=\mathcal{O}(n^{-\sigma-5/2}) 
\quad\mbox{ and }\quad
(-\lap_{\rm N})^{\sigma}_{m,n}g_{n}^{(\sigma)}=\mathcal{O}(n^{-\sigma-3/2})
\]
as $n\to\infty$, which implies claim~(i).

Claim (ii): Let $\sigma>1$. By the Gamma recurrence formula, we find that 
\begin{equation}
    \grad_{\rm D}g_n^{(\sigma)}=\left(\sigma-\frac{1}{2}\right)g_{n}^{(\sigma-1)}
    \label{eq:aux_rec_inproof2}
\end{equation}
for all $n\in\N$. Since $g_{n}^{(\sigma-1)}>0$ for all $n\in\N$, $\grad_{\mathrm{D}}g^{(\sigma)}_n>0$ for all $n\in\N$. 
By claim (i), we also have $\grad_{\mathrm{D}}g^{(\sigma)}\in\mathcal{D}_{\mathrm{N}}^{(\sigma-1)}$.

Claims (iii) and (iv): By~\eqref{eq:lap-dir_entries} and~\eqref{eq:lap-neu_entries}, we have 
\[
(-\lap_{\mathrm{D}})^{\sigma}_{m,n}=c_{m-n}-c_{m+n} 
\quad\mbox{ and }\quad 
(-\lap_{\mathrm{N}})^{\sigma}_{m,n}=c_{m-n}+c_{m+n-1},
\]
with
\[
 c_k:=\frac{(-1)^{k}\Gamma(2\sigma+1)}{\Gamma(\sigma+k+1)\Gamma(\sigma-k+1)},\qquad k\in\Z.
\]
Here, the first term $c_{m-n}$ coincides with the matrix entry $(-\lap)^{\sigma}_{m,n}$ of the Toeplitz operator $(-\lap)^{\sigma}$ studied in \cite{sta-wac_26}. In particular, it is shown in \cite[Lemma~4.2]{sta-wac_26} that
\[
 \sum_{n=1}^{\infty}c_{m-n}\,g^{(\sigma)}_n=\frac{\Gamma^2(\sigma+1/2)}{\pi}\frac{\Gamma(m-1/2)}{\Gamma(m+\sigma)}
\]
for any $m\in\N$. For the Hankel parts, we obtain from the Euler reflection formula~\eqref{eq:euler_reflection} the expressions
\[
 c_{m+n-\kappa}=-\frac{\Gamma(2\sigma+1)\sin(\pi\sigma)}{\pi}\frac{\Gamma(m+n-\kappa-\sigma)}{\Gamma(m+n-\kappa+\sigma+1)}, \quad\kappa\in\{0,1\}.
\]
Then an application of Gauss's hypergeometric theorem~\cite[Eq.~(15.4.20)]{dlmf} in the form
\[
 \sum_{j=0}^{\infty}\frac{\Gamma(a+j)\Gamma(b+j)}{\Gamma(1+j)\Gamma(c+j)}=\frac{\Gamma(a)\Gamma(b)\Gamma(c-a-b)}{\Gamma(c-a)\Gamma(c-b)},
\]
with $c-a-b>0$ and $a,b$ not non-positive integers, yields
\begin{align*}
 \sum_{n=1}^{\infty}c_{m+n-\kappa}\,g^{(\sigma)}_n
 &=-\frac{\Gamma(2\sigma+1)\sin(\pi\sigma)}{\pi}
 \sum_{n=1}^{\infty}\frac{\Gamma(m+n-\sigma-\kappa)\Gamma(n+\sigma-1/2)}{\Gamma(n)\Gamma(m+n+1+\sigma-\kappa)}\\
 &=-\frac{\Gamma^2(\sigma+1/2)\sin(\pi\sigma)}{\pi}\frac{\Gamma(m+1-\sigma-\kappa)}{\Gamma(m+3/2-\kappa)}.
\end{align*}
The choices $\kappa=0$ and $\kappa=1$ give the identities~\eqref{eq:lapD_g_sigma} and~\eqref{eq:lapN_g_sigma}, respectively.

Claim~(v): Since $(\log\Gamma)''>0$ on $(0,\infty)$, see~\cite[Eq.~(5.15.1)]{dlmf}, for every $x>0$, the even function 
$t\mapsto\log\Gamma(x+t)+\log\Gamma(x-t)$
is strictly convex on $[0,x)$, and so the function 
$t\mapsto\Gamma(x+t)\Gamma(x-t)$ is strictly increasing on $[0,x)$.

Let $\sigma\in(0,3/2)$. Taking $x=n+1/2$ and $t=|\sigma-1/2|<1$, we deduce the inequality 
\[
\Gamma(n+\sigma)\Gamma(n+1-\sigma)<\Gamma(n+3/2)\Gamma(n-1/2)
\]
for all $n\in\N$. Since $|\sin(\pi\sigma)|\leq1$, the expression in parentheses in \eqref{eq:lapD_g_sigma} is positive, which implies claim~(v).

Claim~(vi): Let $\sigma\in(0,1/2)$. Arguing as in the proof of (v), with $x=n$ and $t=\sigma<1/2$, which implies
the inequality 
\[
\Gamma(n+\sigma)\Gamma(n-\sigma)<\Gamma(n+1/2)\Gamma(n-1/2)
\]
for all $n\in\N$, we establish positivity of the expression in formula~\eqref{eq:lapN_g_sigma}, and so claim~(vi).
\end{proof}

\section{Proof of the first main result: Theorem~\ref{thm:main}}
\label{sec:proofs}

Throughout this section, we prove Theorem~\ref{thm:main} assuming $1<\sigma<\alpha<3/2$.
Put
\[
 s:=\sigma-1,\qquad p:=2\alpha-2,\qquad
 q:=4\alpha-2\sigma-2.
\]
Thus $0<s<1/2$ and $2s<p<1$. Further, we denote
\begin{equation}
 g:=\mathcal{I}_\alpha,\qquad h:=\grad_{\mathrm{D}}\,\mathcal{I}_\alpha,
 \qquad V:=\frac{(-\lap_\mathrm{N})^sh}{h},
 \label{eq:def_g,h,V}
\end{equation}
and
\begin{equation}
    a_n:=V_{n+1}g_{n+1}g_n, \qquad b_{m,n}:=-(-\lap_{\mathrm{N}})^s_{m,n},
    \label{eq:def_a,b}
\end{equation}
for $m,n\in\N$. By Lemma~\ref{lem:prop_I_alp}(i),(v),
we have $h\in\mathcal{D}_{\mathrm{N}}^{(s)}$ and $h_n>0$ for every
$n\in\N$, and by Lemma~\ref{lem:prop_I_alp}(iii),(v) also
\[
    V_n=\frac{\grad_{\mathrm{D}}\mathcal{I}_{\alpha-\sigma+1}(n)}
 {\grad_{\mathrm{D}}\mathcal{I}_\alpha(n)} >0,\qquad n\in\N,
\]
since $1<\alpha-\sigma+1<3/2$. For brevity, we also define
\[
 Q_{\alpha,\sigma}[u]:=\sum_{n=1}^\infty \overline{u_n} (-\lap_{\mathrm{D}})^\sigma u_n - \sum_{n=1}^\infty W_{\alpha,\sigma}(n)\abs{u_n}^2
\]
for $u\in\ell^{2}(\N)$, where $W_{\alpha,\sigma}=(-\lap_\mathrm{D})^\sigma \mathcal{I}_\alpha / \mathcal{I}_\alpha = \mathcal{I}_{\alpha-\sigma} / \mathcal{I}_\alpha$, which follows from
\eqref{eq:def_I_alp}, \eqref{eq:def_W}, and Lemma~\ref{lem:prop_I_alp}(ii).
Proposition~\ref{prop:two-term-repre} then reads
\begin{equation}\label{eq:main_Q_remainders}
 Q_{\alpha,\sigma}[u]=R_{\mathrm{N}}[u]+R_{\loc}[u], \quad \forall u\in C_0(\N),
\end{equation}
where
\[
R_{\mathrm{N}}[u]=\sum_{1\leq n<m}b_{m,n}h_mh_n
       \abs{\frac{(\grad_{\mathrm{D}}u)_m}{h_m}
              -\frac{(\grad_{\mathrm{D}}u)_n}{h_n}}^2
 \quad\mbox{ and }\quad
 R_{\loc}[u]=\sum_{n\geq1}a_n
       \abs{\frac{u_{n+1}}{g_{n+1}}-\frac{u_n}{g_n}}^2.
\]

We write $x_n\asymp y_n$ when the ratio $x_n/y_n$ is bounded from above and below by positive constants, which may depend only on $\alpha$ and $\sigma$, but not on indices or cut-off parameters. The Gamma ratio asymptotics~\eqref{eq:Gamma_ratio_asymp} used in the respective definitions imply the estimates
\begin{equation}\label{eq:main_comparisons}
 g_n\asymp n^p,\qquad h_n\asymp n^{p-1},\qquad
 V_n\asymp n^{-2s},\qquad a_n\asymp n^q,
 \qquad W_{\alpha,\sigma}(n)\asymp n^{-2\sigma}.
\end{equation}
Furthermore, by using~\eqref{eq:euler_reflection} together with~\eqref{eq:Gamma_ratio_asymp} in~\eqref{eq:lap-neu_entries}, one deduces
\begin{equation}\label{eq:main_kernel_bound}
 0<b_{m,n}\leq C_s|m-n|^{-1-2s} \quad\mbox{ for all } m\neq n,
\end{equation}
where the constant $C_s>0$ depends on $s$ only.

\subsection{Fractional Hardy inequality}\label{subsec:hardy}

By~\eqref{eq:lap_N_neg}, \eqref{eq:I_alp_pos}, and Lemma~\ref{lem:prop_I_alp}(v), $h_n>0$, $a_n>0$, and $b_{m,n}>0$ for all $m,n\in\N$, $m\neq n$. Consequently, both remainders in~\eqref{eq:main_Q_remainders} are non-negative, and hence $Q_{\alpha,\sigma}[u]\geq0$ for all $u\in C_0(\N)$.
By~\eqref{eq:main_comparisons}, the weight $W_{\alpha,\sigma}$ defines a bounded
multiplication operator on $\ell^{2}(\N)$. Since $(-\lap_{\mathrm{D}})^\sigma$ is also bounded (its operator norm equals $4^{\sigma}$),
density extends the inequality $Q_{\alpha,\sigma}[u]\geq0$ to all $u\in\ell^2(\N)$. This proves claim~(i) of Theorem~\ref{thm:main}.

\subsection{Criticality}\label{subsec:crit}

We prove claim~(ii) of Theorem~\ref{thm:main} in three steps. First, we establish a cut-off estimate, which is used in the next step to construct a null-sequence for $Q_{\alpha,\sigma}$, and so prove the criticality of $W_{\alpha,\sigma}$ for $\alpha\leq\alpha_\sigma^*$. The last step shows that $W_{\alpha,\sigma}$ is not critical for $\alpha>\alpha_\sigma^*$.

\emph{Step 1: A cut-off estimate.}
Let $\chi\in C^2(\R)$ be such that $0\leq\chi\leq1$, $\chi=1$ on $(-\infty,1]$ and $\chi=0$ on $[2,\infty)$. For an integer $N\geq2$, $x>0$, and $n\in\N$, define
\[
 \xi_N(x):=\chi\!\left(\frac{\log x}{\log N}\right)
 \quad\mbox{ and }\quad
u_n^N:=g_n\xi_N(n).
\]
In the following, $C$ denotes a generic positive constant depending only on $\alpha$, $\sigma$, and $\chi$.

\begin{lemma}
 For $\alpha\leq\alpha_\sigma^*$ and $N\geq2$, we have the estimate
\begin{equation}\label{eq:main_cutoff_energy}
 Q_{\alpha,\sigma}[u^N]\leq\frac{C}{\log N}.
\end{equation}
\end{lemma}

\begin{proof}
By differentiation, we get
\begin{equation}\label{eq:main_cutoff_derivatives}
 |\xi_N'(x)|\leq\frac{C}{x\log N}
 \quad\mbox{ and }\quad
 |\xi_N''(x)|\leq\frac{C}{x^2\log N}
\end{equation}
for all $x>0$ and $N\geq2$. Denote $\eta_{N}(n):=(\grad_{\mathrm{D}}u^N)_n/h_n$. Then $\eta_N(1)=\xi_N(1)$, and for $n\geq2$, we have
\[
 \eta_{N}(n)=\frac{\xi_{N}(n)g_n-\xi_{N}(n-1)g_{n-1}}{\grad_{\mathrm{D}}g_n}=\xi_{N}(n)+\frac{g_{n-1}}{\grad_{\mathrm{D}}g_n}\bigl(\xi_{N}(n)-\xi_{N}(n-1)\bigr).
\]
Recalling~\eqref{eq:aux_rec_inproof}, this becomes
\begin{equation}\label{eq:main_eta_formula}
 \eta_N(n)=\xi_N(n)+r(n)\bigl(\xi_N(n)-\xi_N(n-1)\bigr),
 \quad\mbox{ with }\quad
 r(x):=\frac{(x-1)(x-\alpha+1)}{(\alpha-1)(2x-1)}.
\end{equation}
For $x\geq2$, let us extend $\eta_N(n)$ to a function $\eta_{N}(x)$ by the expression from~\eqref{eq:main_eta_formula}.
Since $|r(x)|\leq Cx$ and $|r'(x)|\leq C$ for $x\geq2$, \eqref{eq:main_cutoff_derivatives} and
\[
 \xi_N(x)-\xi_N(x-1)=\int_{x-1}^x\xi_N'(t)\dd t,
 \qquad
 \xi_N'(x)-\xi_N'(x-1)=\int_{x-1}^x\xi_N''(t)\dd t
\]
imply
\[
|\eta_N'(x)|\leq \frac{C}{x\log N}
\]
for  $x\geq2$. Since $\chi=1$ on $(-\infty,1]$, we have $\eta_N(1)=\eta_N(2)$. Extend $\eta_N$ constantly to $[1,2]$. Then
\begin{equation}\label{eq:main_eta_increment}
 |\eta_N(m)-\eta_N(n)|\leq\int_{n}^{m}|\eta_N'(x)|\dd x\leq\frac{C}{\log N}\log\frac mn
 \quad\mbox{ for all } 1\leq n<m.
 \end{equation}

Now we may estimate the remainders in~\eqref{eq:main_Q_remainders}. Since both $\xi_N(n)$ and $\eta_N(n)$ vanish for $n\geq N^2+1$, the local remainder is bounded by
\begin{equation}\label{eq:main_local_cutoff}
 R_{\loc}[u^N]=\sum_{n=1}^{N^2+1}a_n\abs{\xi_N(n+1)-\xi_N(n)}^{2}
 \leq\frac{C}{\log^{2}N}\sum_{n=1}^{N^2+1}n^{q-2}\leq\frac{C}{\log N},
\end{equation}
where we used that $|\xi_N(n+1)-\xi_N(n)|\leq C/(n\log N)$, $a_n\leq C n^{q}$, see~\eqref{eq:main_comparisons}, and $q\leq1$, which is equivalent to the assumption $\alpha\leq\alpha_\sigma^*$.

For the non-local remainder, we deduce from \eqref{eq:main_comparisons}, \eqref{eq:main_kernel_bound}, and~\eqref{eq:main_eta_increment} the upper bound
\begin{align*}
 R_{\mathrm{N}}[u^N]&=\sum_{n=1}^{N^{2}+1}\sum_{j=1}^{\infty}b_{n+j,n}h_{n}h_{n+j}\abs{\eta_{N}(n+j)-\eta_{N}(n)}^{2} \\
 &\leq\frac{C}{\log^2N}
 \sum_{n=1}^{N^{2}+1}n^{2p-2s-2}\left[\frac{1}{n}\sum_{j=1}^{\infty}\frac{(1+j/n)^{p-1}}{(j/n)^{1+2s}}\log^2(1+j/n)\right].
\end{align*}
Up to a multiplicative constant, the term in square brackets can be majorised by the integral
\[
\int_{0}^{\infty}\frac{(1+x)^{p-1}}{x^{1+2s}}\log^2(1+x)\dd x,
\]
which converges for $0<s<1/2$ and $2s<p<1$. Consequently, bearing in mind that $2p-2s=q\leq1$, we obtain
\[
R_{\mathrm{N}}[u^N]\leq\frac{C}{\log^{2}N}\sum_{n=1}^{N^{2}+1}n^{q-2}\leq\frac{C}{\log N}.
\]
Together with~\eqref{eq:main_local_cutoff}, this proves \eqref{eq:main_cutoff_energy}.
\end{proof}

\emph{Step 2: Criticality for $\alpha\leq\alpha_\sigma^*$.}
Consider the sequence $u^{N}$ defined in Step~1. Since $\xi_N\to1$ pointwise, $u^N\to g$ pointwise as $N\to\infty$. Moreover, by~\eqref{eq:main_cutoff_energy}, $Q_{\alpha,\sigma}[u^N]\to0$ for $N\to\infty$.
Let $\widetilde{W}\geq0$ satisfy
\[
Q_{\alpha,\sigma}[u]\geq\sum_{n=1}^{\infty}\widetilde{W}_{n}|u_n|^{2}
\]
for all $u\in C_0(\N)$. Then, for every fixed $n\in\N$,
\[
 0\leq \widetilde{W}_{n}|u_n^N|^{2}\leq Q_{\alpha,\sigma}[u^N].
\]
Sending $N\to\infty$, we get $\widetilde{W}_n|g_n|^{2}=0$ for every $n\in\N$.
Since $g_n>0$, it follows that $\widetilde{W}_n=0$ for every $n\in\N$.
This proves the criticality of $W_{\alpha,\sigma}$ for $\alpha\leq\alpha_\sigma^*$.

\emph{Step 3: Non-criticality for $\alpha>\alpha_\sigma^*$.}
Suppose now that $\alpha>\alpha_\sigma^*$, i.e. $q>1$. Then, by~\eqref{eq:main_comparisons}, we have
\[
 A:=\sum_{n=1}^\infty \frac{1}{a_n}<\infty.
\]
For $u\in C_0(\N)$, Cauchy--Schwarz gives
\[
 \abs{\frac{u_1}{g_1}}^2=\abs{\sum_{n=1}^\infty\left(\frac{u_n}{g_n}-\frac{u_{n+1}}{g_{n+1}}\right)}^2
 \leq A\sum_{n=1}^\infty a_n\abs{\frac{u_n}{g_n}-\frac{u_{n+1}}{g_{n+1}}}^{2} =A R_{\loc}[u].
\]
It follows that
\begin{equation}\label{eq:main_rank_one_improvement}
 Q_{\alpha,\sigma}[u]\geq\frac{|u_1|^2}{A g_1^2}.
\end{equation}
By density and boundedness, \eqref{eq:main_rank_one_improvement}
extends to every $u\in\ell^2(\N)$. Thus
\[
    \widetilde W_n
    :=W_{\alpha,\sigma}(n)+\frac{\delta_{n,1}}{A g_1^2}
\]
is a Hardy weight strictly larger than $W_{\alpha,\sigma}$
at $n=1$. This proves non-criticality and completes
the proof of claim~(ii) of Theorem~\ref{thm:main}.

\subsection{Non-attainability}\label{subsec:non-att}

We prove non-attainability in $\ell^{2}(\N)$ of $W_{\alpha,\sigma}$ for every $1<\sigma<\alpha<3/2$.
Suppose $u\in\ell^2(\N)$ satisfies $Q_{\alpha,\sigma}[u]=0$, and approximate it by finitely supported sequences $u^{(k)}\to u$ in $\ell^2(\N)$ as $k\to\infty$. The boundedness of both operators $(-\lap_{\mathrm{D}})^{\sigma}$ and $W_{\alpha,\sigma}$ implies $Q_{\alpha,\sigma}[u^{(k)}]\to0$ for $k\to\infty$. For each fixed $n\in\N$, we infer from \eqref{eq:main_Q_remainders} that
\[
 a_n\abs{\frac{u_{n+1}^{(k)}}{g_{n+1}}
            -\frac{u_n^{(k)}}{g_n}}^2 \leq Q_{\alpha,\sigma}[u^{(k)}].
\]
Passing to the limit $k\to\infty$ yields $u_{n+1}/g_{n+1}=u_n/g_n$ for all
$n\geq1$. Thus, $u=cg$ for some $c\in\C$. Since $g_n\asymp n^p$ with $p>0$, $u\in\ell^2(\N)$ only if $c=0$, i.e. $u\equiv0$. This proves claim~(iii) of Theorem~\ref{thm:main}.

\subsection{Optimality at the critical parameter}\label{subsec:opt}

To prove claim (iv) of Theorem~\ref{thm:main}, it remains to show the optimality near infinity of $W_{\alpha,\sigma}$ for $\alpha=\alpha_\sigma^*$. So assume $\alpha=\alpha_\sigma^*$, which means $q=1$. Then, from~\eqref{eq:main_comparisons}, we have
\begin{equation}\label{eq:main_null_critical_mass}
 W_{\alpha,\sigma}(n)g_n^2\asymp n^{2p-2\sigma}=n^{-1}.
\end{equation}

Now we choose a function $\chi\in C^2(\R)$ with $0\leq\chi\leq1$ such that $\chi=0$ on
$(-\infty,1]\cup[4,\infty)$ and $\chi=1$ on $[2,3]$.
The same cut-off construction as in Step~1 of Subsection~\ref{subsec:crit} gives $u_n^N=0$ for $n\leq N$ and
$u_n^N=g_n$ for $N^2\leq n\leq N^3$. In the estimates of the
remainders, the outer sums now run up to $N^4+1$.
Since $\sum_{n=1}^{N^4+1}n^{-1}\leq C\log N$, the same argument
shows that~\eqref{eq:main_cutoff_energy} remains true, i.e.
\[
 Q_{\alpha,\sigma}[u^N]\leq\frac{C}{\log N}
\]
for all $N\geq2$. On the other hand, we infer from~\eqref{eq:main_null_critical_mass} that
\[
  \sum_{n=1}^{\infty}W_{\alpha,\sigma}(n)|u_n^N|^2\geq\sum_{n=N^2}^{N^3}W_{\alpha,\sigma}(n)g_n^2\geq C\sum_{n=N^2}^{N^3}\frac1n\geq C \log N.
\]
Therefore
\[
 \frac{\sum_{n=1}^\infty \overline{u_n^{N}} (-\lap_{\mathrm{D}})^\sigma u_n^{N}}
      {\sum_{n=1}^{\infty}W_{\alpha,\sigma}(n)|u_n^N|^2}
 =1+\frac{Q_{\alpha,\sigma}[u^N]}{\sum_{n=1}^{\infty}W_{\alpha,\sigma}(n)|u_n^N|^2}
 \longrightarrow1, \;\mbox{ as } N\to\infty.
\]
Since the supports of $u^N$ eventually avoid every finite set, this proves the optimality near infinity. The proof of Theorem~\ref{thm:main} is complete.
\qed

\section{Proof of the second main result: Theorem~\ref{thm:new_opt_weight}}
\label{sec:new_weight}

In this section, we prove Theorem~\ref{thm:new_opt_weight} using Proposition~\ref{prop:two-term-repre} for $\sigma>1$, and Lemma~\ref{lem:frac_lap_N_id} for $\sigma\leq1$, with the sequence $g$ given by~\eqref{eq:def_g_sigma}. As the proof proceeds similarly to that of Theorem~\ref{thm:main}, we keep it brief.

Put $g:=g^{(\sigma)}$, $p:=\sigma-1/2$, and
\[
 Q_\sigma[u]:=\sum_{n=1}^\infty \overline{u_n} (-\lap_{\mathrm{D}})^\sigma u_n - \sum_{n=1}^\infty \widehat{W}_{\sigma}^{\rm op}(n)\abs{u_n}^2.
\]
By Lemma~\ref{lem:lap_g_sigma}, $g\in\mathcal{D}_{\mathrm{D}}^{(\sigma)}$ and $\widehat{W}_{\sigma}^{\rm op}(n)=(-\lap_{\mathrm{D}})^\sigma g_n/g_n>0$ for all $n\in\N$, where one needs to apply the duplication formula~\cite[Eq.~(5.5.5)]{dlmf}
\[
    \Gamma(2z)=\frac{2^{2z-1}}{\sqrt{\pi}}\Gamma(z)\Gamma\bigg(z+\frac{1}{2}\bigg).
\]
From the explicit formulas \eqref{eq:new_opt_weight}, \eqref{eq:def_g_sigma}, and~\eqref{eq:Gamma_ratio_asymp}, one readily deduces the bounds
\begin{equation}\label{eq:new_comparisons}
 g_n\asymp n^{p},\qquad \widehat{W}_{\sigma}^{\rm op}(n)\asymp n^{-2\sigma},\qquad \widehat{W}_{\sigma}^{\rm op}(n)g_n^2\asymp n^{-1}.
\end{equation}

\subsection{Case \texorpdfstring{$1<\sigma<3/2$}{1 < sigma < 3/2}} Let $s:=\sigma-1$ and $h$, $V$, $a$, and $b$ be defined as in Section~\ref{sec:proofs}, see \eqref{eq:def_g,h,V} and \eqref{eq:def_a,b}, but with $g=g^{(\sigma)}$ in place of $\mathcal{I}_{\alpha}$. 
By Lemma~\ref{lem:lap_g_sigma}(ii), $h\in\mathcal{D}_{\mathrm{N}}^{(s)}$ and $h_n>0$ for all $n\in\N$. Moreover, bearing the identity~\eqref{eq:aux_rec_inproof2} in mind, Lemma~\ref{lem:lap_g_sigma}(vi) implies that 
\[
 V_n=\frac{(-\lap_{\mathrm{N}})^{s}h_n}{h_n}=\frac{(-\lap_{\mathrm{N}})^{s}g^{(s)}_n}{g^{(s)}_n}>0
\]
for all $n\in\N$. Hence Proposition~\ref{prop:two-term-repre} applies, and $Q_\sigma[u]=R_{\mathrm{N}}[u]+R_{\loc}[u]$ for $u\in C_0(\N)$ with non-negative remainders as in~\eqref{eq:main_Q_remainders}. 

Moreover, it is straightforward to check that we have the same bounds as in \eqref{eq:main_comparisons} with $W_{\alpha,\sigma}$ replaced by $\widehat{W}_\sigma^{\rm op}$, $p=\sigma-1/2$, and $q=1$. The kernel bound~\eqref{eq:main_kernel_bound} does not depend on $g$, and the function $r$ in~\eqref{eq:main_eta_formula} is now $r(x)=(x-1)/(\sigma-1/2)$, because $g_{n-1}/h_n=(n-1)/(\sigma-1/2)$; it again satisfies $|r(x)|\leq Cx$ and $|r'(x)|\leq C$. With these replacements, the arguments of Subsections~\ref{subsec:hardy}, \ref{subsec:non-att}, and~\ref{subsec:opt}, and of Steps~1 and~2 of Subsection~\ref{subsec:crit} (with $q=1$) apply verbatim.

\subsection{Case \texorpdfstring{$0<\sigma\leq1$}{0 < sigma <= 1}} Lemma~\ref{lem:frac_lap_N_id} with $\mathrm{X}=\mathrm{D}$ and $h=g$ yields
\[
 Q_\sigma[u]=R_{\mathrm{D}}[u]:=\sum_{1\leq n<m}b^{\mathrm{D}}_{m,n}\,g_mg_n\abs{\frac{u_m}{g_m}-\frac{u_n}{g_n}}^2,\qquad u\in C_0(\N),
\]
where $b^{\mathrm{D}}_{m,n}:=-(-\lap_{\mathrm{D}})^{\sigma}_{m,n}$ satisfies
\[
 0\leq b^{\mathrm{D}}_{m,n}\leq C|m-n|^{-1-2\sigma}\quad\mbox{ for } m\neq n,
 \qquad b^{\mathrm{D}}_{n+1,n}>0,
\]
see~\cite[Proposition~2.2]{das-fer_26}. Note that for $\sigma=1$, we have $b^{\mathrm{D}}_{m,n}=\delta_{|m-n|,1}$ for all $m\neq n$.
Thus $Q_\sigma\geq0$ on $C_0(\N)$, and the Hardy inequality follows by density as in Subsection~\ref{subsec:hardy}. Non-attainability follows as in Subsection~\ref{subsec:non-att}, with $a_n$ replaced by $b^{\mathrm{D}}_{n+1,n}g_{n+1}g_n>0$; note that $g\notin\ell^2(\N)$, since we have $2p=2\sigma-1>-1$ in \eqref{eq:new_comparisons}. By~\eqref{eq:new_comparisons}, the arguments of Step~2 of Subsection~\ref{subsec:crit} and of Subsection~\ref{subsec:opt} give criticality and optimality near infinity, once we show that $Q_\sigma[u^N]\leq C/\log N$ for $u^N=g\,\xi_N$ with both choices of $\chi$ used there. If $\sigma=1$, then the non-local remainder
\[
    R_{\mathrm{D}}[u^N]=\sum_{n=1}^\infty g_{n+1}g_n|\xi_N(n+1)-\xi_N(n)|^2
\]
is estimated as $R_{\loc}$ in~\eqref{eq:main_local_cutoff}, since $g_{n+1}g_n\asymp n$. If $\sigma<1$, then
\[
    |\xi_N(m)-\xi_N(n)|\leq \frac{C}{\log{N}}\log\frac{m}{n} \quad\mbox{ for } 1\leq n<m
\]
by~\eqref{eq:main_cutoff_derivatives}, and the estimate of $R_{\mathrm{N}}$ in Step~1 of Subsection~\ref{subsec:crit}, with $(h,s,\eta_N)$ replaced by $(g,\sigma,\xi_N)$, yields
\[
 R_{\mathrm{D}}[u^N]\leq\frac{C}{\log^2N}\sum_{n=1}^{N^4}n^{2p-2\sigma}\left[\frac1n\sum_{j=1}^{\infty}\frac{(1+j/n)^{p}}{(j/n)^{1+2\sigma}}\log^2(1+j/n)\right]\leq\frac{C}{\log N},
\]
because $2p-2\sigma=-1$ and the integral $\int_0^\infty(1+x)^{p}x^{-1-2\sigma}\log^2(1+x)\dd x$ converges for $0<\sigma<1$ and $p<2\sigma$.
\qed

\subsection*{Acknowledgement}

This research was supported by the CTU Future Fund (Project ID: CVUT-BrF-26-23028S).

\subsection*{AI-use disclosure}

While the core ideas and theoretical foundation of this article stem from our previous AI-unassisted work \cite{sta-wac_26}, AI tools (ChatGPT and Claude) were regularly consulted during the writing process of this manuscript. Specifically, these tools were used for text editing, mathematical proof assistance and verification, and improving the clarity of exposition. The final manuscript was authored, verified, and approved by the human authors, who bear full responsibility for its content.

\subsection*{Data availability}

Data sharing is not applicable to this article as no datasets
were generated or analysed during the current study.

\appendix
\section{A lower bound for the optimal Hardy weight}
\label{app:weight_lower_bound}

We conclude the paper with the proof of inequality \eqref{eq:W-lower-bound}.

\begin{proposition}
For every $\sigma\in(0,3/2)$ and $n\in\N$, we have
\[
    W_\sigma^{\mathrm{op}}(n)>
      4^\sigma\frac{\Gamma^2(\frac{3+2\sigma}{4})}{\Gamma^2(\frac{3-2\sigma}{4})}\frac{1}{n^{2\sigma}}.
\]
\end{proposition}

\begin{proof}
Let $\sigma\in(0,3/2)$. Recalling formula~\eqref{eq:wop_explicit}, the desired inequality can be reformulated as
\begin{equation}
    n^{2\sigma}\,\frac{\Gamma(n-\frac{1+2\sigma}{4})\Gamma(n+\frac{5-2\sigma}{4})}{\Gamma(n-\frac{1-2\sigma}{4})\Gamma(n+\frac{5+2\sigma}{4})}
    >1, \quad n\in\N.
\label{eq:ineq_aux_append}
\end{equation}

Put $a:=(1+2\sigma)/4$, $b:=(1-2\sigma)/4$, and define the auxiliary function
\[
    F(x):=\log\left(x^{2\sigma}\,\frac{\Gamma(x-a)\Gamma(x+1+b)}{\Gamma(x-b)\Gamma(x+1+a)}\right)\quad \mbox{ for } x>a. 
\]
All Gamma arguments are positive for $x>a>b$.
Moreover, the Gamma ratio asymptotics \eqref{eq:Gamma_ratio_asymp} imply that $F(x)\to0$ for $x\to\infty$. In the rest of the proof, we show that $F'(x)<0$, which implies $F(x)>0$ for all $x>a$. The inequality \eqref{eq:ineq_aux_append} follows since $a<1$.

In the following, $\psi:=\Gamma'/\Gamma$ is the digamma function. 
Using the integral identity \cite[Eq.~(5.9.12)]{dlmf}
\[
    \psi(v)-\psi(u)=\int_0^\infty\frac{e^{-ut}-e^{-vt}}{1-e^{-t}}\,\dd t,
    \quad u,v>0,
\]
we obtain
\begin{align*}
    F'(x)&=\frac{2\sigma}{x}+\psi(x-a)-\psi(x-b)+\psi(x+1+b)-\psi(x+1+a)\\
    &=\int_0^\infty e^{-xt}\left(2\sigma-\frac{e^{at}-e^{bt}+e^{-(1+b)t}-e^{-(1+a)t}}{1-e^{-t}}\right)\dd t\\
    &=2\int_0^\infty e^{-xt}\left(\sigma-\frac{\cosh(3t/4)\sinh(\sigma t/2)}{\sinh(t/2)}\right)\dd t.
\end{align*}
These integrals converge absolutely: the expression in parentheses in the last integral extends continuously to $t=0$, and is $\mathcal{O}(e^{at})$ as $t\to\infty$, whereas $x>a$. Using, for $y>0$, the inequalities
\[
    \sinh(\sigma y)>\sigma y
    \quad\mbox{ and }\quad
    \sinh y=\int_0^y\cosh r\,\dd r <y\cosh y<y\cosh(3y/2),
\]
we may estimate
\[
    \frac{\cosh(3y/2)\sinh(\sigma y)}{\sinh y}
    >\sigma.
\]
Taking $y=t/2$, we see that the integrand that defines $F'(x)$ is strictly negative for all $t>0$.
Consequently, $F'(x)<0$ for all $x>a$, as claimed. The proof is complete.
\end{proof}

\bibliographystyle{acm}

\end{document}